\documentclass[11pt, reqno]{amsart}
\usepackage{amsmath, amsthm, amssymb, amsfonts, mathrsfs}
\usepackage{verbatim} 
\usepackage{url}
\usepackage{hyperref}
\usepackage{xcolor}
\hypersetup{
    colorlinks=true,       
    linkcolor=blue,          
    citecolor=cyan,        
}

\usepackage{mathtools}

\usepackage{tikz-cd}

\usepackage{enumitem}
\setitemize{itemsep=1pt}
\setenumerate{itemsep=1pt}

\theoremstyle{plain}
\newtheorem{theorem}{Theorem}
\newtheorem{lemma}[theorem]{Lemma}
\newtheorem{corollary}[theorem]{Corollary}
\newtheorem{proposition}[theorem]{Proposition}
\newtheorem{conjecture}[theorem]{Conjecture}

\theoremstyle{remark}
\newtheorem{remark}{Remark}

\makeatletter

\newcommand{\Rmnum}[1]{\expandafter\@slowromancap\romannumeral #1@}
\makeatother

\def\ri{\mathrm i}
\def\rd{\,\mathrm d}

\def\rb{\mathbb R}
\def\nb{\mathbb N}
\def\zb{\mathbb Z}
\def\cb{{\mathbb C}}

\numberwithin{equation}{section}
\allowdisplaybreaks[2] 

\author{Michael J.\ Schlosser}
\address{Fakult\"at f\"ur Mathematik, Universit\"at Wien,
Oskar-Morgenstern-Platz~1, A-1090 Vienna, Austria}
\email{michael.schlosser@univie.ac.at}
\thanks{The first author was partially supported by FWF Austrian Science Fund
grant P 32305.}

\author{Nian Hong Zhou}
\address{School of Mathematics and Statistics,
Guangxi Normal University,
No.1 Yanzhong Road, Yanshan District, Guilin, 541006,
Guangxi, PR China}
\email{nianhongzhou@outlook.com, nianhongzhou@gxnu.edu.cn}
\thanks{The second author was partially supported by Guangxi Science and Technology Plan Project \#2020AC19236.}

\title[On the infinite Borwein product]{On the infinite Borwein product\\
raised to a positive real power}

\subjclass[2010]{Primary 11P55; Secondary 11F03, 11F30, 26D20}

\keywords{infinite Borwein product, sign pattern, asymptotics, positivity,
circle method, vanishing of coefficients.}

\dedicatory{Dedicated to the memory of Richard Allen Askey}

\begin{document}

\begin{abstract}
In this paper, we study properties of the coefficients appearing in the
$q$-series expansion of $\prod_{n\ge 1}[(1-q^n)/(1-q^{pn})]^\delta$,
the infinite Borwein product for an arbitrary prime $p$, raised
to an arbitrary positive real power $\delta$.
We use the Hardy--Ramanujan--Rademacher circle method
to give an asymptotic formula for the coefficients. For $p=3$ we give
an estimate of their growth which enables us to partially confirm
an earlier conjecture of the first author concerning an observed
sign pattern of the coefficients when the exponent $\delta$ is within
a specified range of positive real numbers.
We further establish some vanishing and divisibility properties of the
coefficients of the cube of the infinite Borwein product.
We conclude with an Appendix presenting several new conjectures on precise
sign patterns of infinite products raised to a real power which are similar
to the conjecture we made in the $p=3$ case.

\end{abstract}

\maketitle

\section{Introduction and statement of results}

Let $q$ be a complex number with $0<|q|<1$. Define
\begin{equation}\label{eq:f}
f(q)=\prod_{n\ge 1}\frac{1}{1-q^n},
\end{equation}
and, for $p$ being a prime,
\begin{equation}\label{eq:Gp}
  G_p(q)=\frac{f(q^p)}{f(q)}.
\end{equation}
We shall call $G_p(q)$ the \emph{infinite Borwein product}.
It is well known that $f(q)$ is the generating function for the number of
unrestricted partitions $p(n)$, that is
$$
f(q)=\sum_{n\ge 0}p(n)q^n.
$$
Using the modularity of $f(q)$, Hardy and Ramanujan \cite{MR1575586}
and Rademacher \cite{MR8618} proved that
\begin{equation}\label{part-asym}
  p(n)=\frac{1}{\pi\sqrt{2}}
  \sum_{k\ge 0}k^{1/2}\sum_{\substack{h\!\!\pmod{k}\\
      \gcd(h,k)=1}}\omega_{h,k}e^{-\frac{2\pi\ri hn}{k}}
  \frac{\,d}{\,dn}\frac{\sinh\Big(\frac{\pi}{k}\sqrt{\frac{2}{3}(n-1/24)}\Big)}
  {\sqrt{n-1/24}},
\end{equation}
for all integers $n\ge 1$. Here and throughout this paper,
\begin{equation}\label{om}
  \omega_{h,k}=e^{\pi\ri s(h,k)},
\end{equation}
  with $s(h,k)$ being the Dedekind sum
\begin{equation}\label{Dedsum}
  s(h,k)=\sum_{1\le j<k}\bigg(\frac{j}{k}-
    \left\lfloor\frac{j}{k}\right\rfloor-
    \frac{1}{2}\bigg)\bigg(\frac{jh}{k}-
    \left\lfloor\frac{jh}{k}\right\rfloor-\frac{1}{2}\bigg).
\end{equation}

According to Andrews \cite{MR1395410}, P.~Borwein considered the
$q$-series expansion
\begin{equation}\label{eqib}
G_{p}(q)=\sum_{n\ge 0}c_{p}(n)q^n,
\end{equation}
as part of an unpublished study of modular forms.
While it is clear from \eqref{eq:Gp} that $G_p(q)^{-1}$ is the
generating function for partitions into parts that are not a multiple of $p$,
and thus has non-negative coefficients,
the coefficients $c_p(n)$ in \eqref{eqib} have different signs.
Andrews~\cite[Theorem~2.1]{MR1395410} proved the following
result, and noted that Garvan and Borwein have a different proof in
unpublished work of 1990.
\begin{theorem}\label{thm:a} For all primes $p$, $c_p(n)$ and $c_p(n+p)$
have the same sign for each $n\ge 0$, i.e.,
$$c_{p}(n)c_{p}(n+p)\ge 0,$$
for each $n\ge 0$.
\end{theorem}
We say that the coefficients $c_p(n)$ have a sign pattern of period $p$.

In September 2019, as a result of experimentation using computer algebra,
the first author of the present paper presented a
conjecture~\cite[Conjecture~1]{Schl2019}
in a tribute dedicated to Richard Askey.
We reproduce this conjecture in Conjecture~\ref{conj1} below;
one of the main results of this paper is a partial affirmation of it,
see Corollary~\ref{cor}.
\begin{conjecture}\label{conj1}
Let $\delta$ be a real number satisfying
\begin{equation*}
0.227998127341\ldots\approx\frac{9-\sqrt{73}}2\le \delta\le 1\quad\text{or}
\quad 2\le \delta\le 3.
\end{equation*}
Then the series $A^{(\delta)}(q)$, $B^{(\delta)}(q)$, $C^{(\delta)}(q)$
appearing in the dissection
\begin{equation*}
G_3(q)^\delta=
A^{(\delta)}(q^3)-qB^{(\delta)}(q^3)-q^2C^{(\delta)}(q^3)
\end{equation*}
are power series in $q$ with non-negative real coefficients.
\end{conjecture}

With other words, for the exponent $\delta$ within the specified range
of real numbers the $q$-series coefficients of $G_3(q)^\delta$
exhibit the sign pattern $+--$.

We present several similar conjectures on precise sign patterns
for other infinite products raised to a power within specified ranges
of real numbers in Appendix~\ref{sec:app}.

The validity of Conjecture~\ref{conj1} for $\delta=1$ is known and easy
to prove by using Jacobi's triple product identity,
see e.g.\ \cite{St1999}. For $\delta=3$, we actually have a
result for any prime $p$, not only for $p=3$, see Theorem~\ref{thm:a3}.

It is actually not difficult to explain why the condition
$\delta\in[\frac{9-\sqrt{73}}2,1]\cup[2,3]$
(leaving out the trivial case $\delta=0$) is \textit{necessary}
for the sign-pattern $+--$ to hold.
In fact, we have the Taylor
series expansion (which is routine to compute using any
computer algebra system)
\begin{align*}
  G_3(q)^\delta
  &=1-\delta q+\frac{\delta(\delta-3)}2 q^2-
    \frac{\delta(\delta^2-9\delta+2)}6 q^3
    +\frac{\delta(\delta^3-18\delta^2+35\delta-42)}{24} q^4\\*
  &\quad\,
  -\frac{\delta(\delta-1)(\delta-2)(\delta-3)(\delta-24)}{120} q^5
  +O(q^6).
\end{align*}
For the sign pattern $+--$ to hold, first of all the coefficient of $q^1$
in $G_3(q)^\delta$ should be non-positive. This implies $\delta>0$.
(We excluded the trivial case $\delta=0$ in the first place.)
The coefficient of $q^2$ should
be non-positive as well. This forces $0<\delta\le 3$.
We turn to the coefficient of $q^3$. The two roots of $\delta^2-9\delta+2$
are $\frac{9\pm\sqrt{73}}2$ and it is easy to see that
the coefficient of $q^3$ can only be non-negative
if $\frac{9-\sqrt{73}}2\le\delta\le\frac{9+\sqrt{73}}2$.
Since $\delta\le 3$ (from before) we have reached the point that we need
$\frac{9-\sqrt{73}}2\le\delta\le 3$.
Finally, for the coefficient of $q^5$
(we don't need to consider the coefficient of $q^4$ here)
to be non-positive we obviously  need to exclude $1<\delta<2$.
Altogether we have explained the necessity of the
specified range of real numbers for $\delta$.
The surprising fact is that this range is also (conjectured to be)
\textit{sufficient} for all of the coefficients to satisfy
the sign pattern $+--$.

While Conjecture~\ref{conj1} concerns a statement about
a sign-pattern that holds from the first coefficient on
for suitably restricted $\delta>0$, we actually believe that the
sign pattern $+--$ holds for any $\delta>0$ in an
asymptotic sense, namely from the $n$-th coefficient on,
where $n$ is an integer depending on $\delta$.

This serves as our motivation to apply an asymptotic approach
towards settling Conjecture~\ref{conj1} where we initially
just assume $\delta>0$
(not further restricted), and only later restrict
$\delta$ to be within specified intervals when desired.
To achieve our goal we shall employ the Hardy--Ramanujan
circle method perfected by Rademacher~\cite{MR8618}
(see also \cite[Chapter~14]{MR0364103}).
This will enable us to give an asymptotic formula for the $q$-series
coefficients $c_p^{(\delta)}(n)$ appearing in the infinite Borwein product
raised to a real power $\delta>0$, i.e.\ of
\begin{equation}
G_p(q)^{\delta}=\sum_{n\ge 0}c_p^{(\delta)}(n)q^n,
\end{equation}
where $p$ is any prime (not necessarily $p=3$).
We shall refer to the $c_p^{(\delta)}(n)$ as \emph{Borwein coefficients}.

At this point it is appropriate to mention that the use of asymptotic machinery
to prove positivity results (including sign patterns) for the coefficients
appearing in infinite $q$-products is quite established and known to be
efficient.
In particular, Richmond and Szekeres~\cite{MR515217}, making heavy use of
results of Iseki~\cite{MR95807,MR0108473},
employed the Rademacher circle method to
prove the sign pattern of the G\"ollnitz--Gordon continued fraction.
Recently, Chern~\cite{Ch2019,Ch2020} established the asymptotics of the
coefficients of any finite product of Dedekind eta functions, and similarly
the asymptotics of the coefficients of $q$-products satisfying
modular symmetries.
His methods are very similar to those we use in the present paper
but we consider arbitrary real powers of the infinite products
(and our applications are of a different, more analytic nature).
We would also like to mention that
C.~Wang~\cite{Wang2019} recently utilized asymptotic machinery to
settle the famous first Borwein Conjecture (cf.\ \cite{MR1395410})
which is a statement about the coefficients appearing in a sequence of
\textit{finite} products.
Some related open conjectures about sequences of infinite
products were recently raised by Bhatnagar and the first author in
\cite{BhSchl2019}, however, no attempt was made there to attack the
conjectures by asymptotic machinery or by other means.

\smallskip
In order to state our results we recall the definition
of the modified Bessel function of the first kind $I_1(z)$ given by
\begin{equation}\label{eqbs}
I_1(z):=\sum_{n\ge 0}\frac{1}{n!(n+1)!}\Big(\frac{z}{2}\Big)^{2n+1},
\end{equation}
cf.\ \cite[p.~222, Equation~(4.12.2)]{AAR}, which is an entire function.
Its integral representation is
\begin{equation}\label{eq:intrpI1}
  I_1(z)=\frac{(z/2)}{2\pi\ri}\int_{1-\ri\infty}^{1+\ri\infty}
  e^{w+z^2/4w}w^{-2}\rd w,
\end{equation}
cf.\ \cite[p.~236, Exercise~13]{AAR}.

For any prime $p$, we have the following asymptotic formula
(of arbitrary positive integer order $N$) for the Borwein
coefficients $c_p^{(\delta)}(n)$,
where $\delta$ is within a specified range of positive real numbers
depending on $p$. (Recall that, according to \eqref{om} and \eqref{Dedsum},
$\omega_{h,k}$ denotes certain exponentials of Dedekind sums.)
\begin{theorem}\label{mth}
  Let $\delta\in(0,24/(p-1)]$ and let $N\in\nb$. For each integer $n\ge 1$
  we have
\begin{align*}
c_p^{(\delta)}(n)&=\frac{2\pi \delta^{1/2}}{\sqrt{\frac{24n}{p-1}-\delta}}
\sum_{1\le k\le N }A_{pk}^{(\delta)}(n)I_1\Bigg(\frac{(p-1)\pi}{6pk}
\sqrt{\delta\bigg(\frac{24n}{p-1}-\delta\bigg)}\Bigg)\\*
&\quad\;+e^{\frac{(24n-(p-1)\delta)\pi }{6p^2N^2}}E_{p,N}^{(\delta)}(n),
\end{align*}
where
\begin{equation*}
A_k^{(\delta)}(n)=\frac{1}{k}\sum_{\substack{0\le h<k\\ \gcd(h,k)=1}}
\Big(\omega_{h,k}^{-1}\omega_{h,\frac{k}{p}}\Big)^{\delta}
e^{-\frac{2\pi\ri hn}{k}}.
\end{equation*}
Further, the error term $E_{p,N}^{(\delta)}(n)$ satisfies the bound
\begin{align*}
\Big|E_{p,N}^{(\delta)}(n)\Big|&\le\frac{(p-1)e^{\frac{(p-1)\pi\delta}{12}} }
{p^2}\Big(\pi\sqrt{2}-2+
2f\big(e^{-6\pi }\big)^{\delta}f\big(e^{-2\pi}\big)^{\delta}\Big)\\*
&\quad\;+ \frac{ 2(p-1)\cdot e^{-\frac{\pi(p-1)\delta}{12p}}}
{p^{1-\delta/2}}f\Big(e^{-\frac{2\pi }{p}}\Big)^{\delta}
f\big(e^{-2\pi}\big)^{\delta}.
\end{align*}
\end{theorem}

\begin{remark}
Throughout this paper, $z^{a}:=e^{a\log z}$ and the logarithms
are always understood to assume their principal values,
that is $\arg(z)\in[-\pi, \pi)$.
\end{remark}
\begin{remark}
The right-hand side of the inequality for $\big|E_{p,N}^{(\delta)}(n)\big|$
in Theorem~\ref{mth} is independent from $N$ and $n$,
thus the error term $E_{p,N}^{(\delta)}(n)$ is $O(1)$.
By using a similar argument to that of Rademacher and Zuckerman in
their proof of \cite[Theorem~1]{MR1503417}, we can extend the
specified region for $\delta$ in Theorem~\ref{mth}
(which is $\delta\in(0,24/(p-1)]$) to all $\delta>0$,
still with an $O(1)$ error term.
However, the expressions for the main term and the effective error term
are then more complicated. Since we are mainly interested in the
asymptotics in certain confined regions (after all, our main aim
concerns the development of tools to understand and tackle concrete
observations such as those in Conjecture~\ref{conj1} and similar
conjectures in Appendix~\ref{sec:app}), we leave the details
of using Rademacher and Zuckerman's method to to extend the range of
$\delta$ to all positive reals to the interested reader.
\end{remark}

Focusing on the case $p=3$, we can use Theorem~\ref{mth} to give
the following growth estimate for the Borwein coefficients $c_3^{(\delta)}(n)$
for $\delta$ within a specified range.

\begin{theorem}\label{mth1}
Let $\delta\in[0.227,3]$, and define
\begin{align*}
\hat{c}_3^{(\delta)}(n)&
 =\frac{2\pi \delta^{1/2}}{3\sqrt{12n-\delta}}\,
I_1\Big(\frac{\pi}{9}\sqrt{\delta\left(12n-\delta\right)}\Big),\\*
L_{\delta,n}&=\frac{\pi}{18}\sqrt{\delta\left(12n-\delta\right)},
\end{align*}
and
$$w(\delta)=\frac{1}{2}\log\bigg(\frac{1}{\delta}\bigg)+
\frac{0.736\,\big(1.689^{\delta}\big(1.222+1.002^{\delta}\big)+
    3\cdot 1.692^{\delta}\big)}{\delta}+0.119. $$
Then we have for all $n\in\nb$ the inequality
\begin{equation*}
\Bigg|\frac{c_3^{(\delta)}(n)}{\hat{c}_3^{(\delta)}(n)}
- \cos\bigg(\frac{\pi\delta}{18}+\frac{2\pi n}{3}\bigg)\Bigg|\le
\frac{L_{\delta,n} w(\delta)+L_{\delta,n}\log L_{\delta,n}
+2 I_1(L_{\delta,n})}{I_1(2L_{\delta,n})}.
\end{equation*}
\end{theorem}
We obtained the numerical constants appearing in the expression
for $w(\delta)$ with the aid of \emph{Mathematica};
a strengthening of the result with a higher precision of the involved
constants is a question of computational resources (suitable software,
running time and memory).
In principle, Theorem~\ref{mth} would even enable us to give
precise growth estimates for the Borwein coefficients $c_3^{(\delta)}(n)$
for $\delta$ within the larger range $[\epsilon,12]$ where $\epsilon$
is any given positive real number.
Now our experimentation using \textit{Mathematica} showed that
the computations converge considerably faster for
$\epsilon\le\delta\le 3$ where $\epsilon$ is not much less than
$\frac{9-\sqrt{73}}{2}$ than outside this region (which
is not a big surprise in view of Conjecture~\ref{conj1}).
Therefore, for practical computational reasons we took $\epsilon=0.227$
and restricted the initial range $(0,12]$ (coming from the $p=3$ case of
Theorem~\ref{mth}) to the range $[0.227,3]$ which is still larger
than the range for $\delta$ specified in Conjecture~\ref{conj1},
namely $[\frac{9-\sqrt{73}}2,1]\cup[2,3]$ which is the range we mainly
care about. At this point we would like to remind
the reader that Conjecture~\ref{conj1} concerns an assertion about the
\textit{precise} behavior of the coefficients of a series while
Theorem~\ref{mth1} (and the following Corollary~\ref{cor})
concerns their \textit{asymptotic} behavior.

Finally, we are able to partially affirm Conjecture~\ref{conj1}
(again, with numerical constants obtained with the aid of
\emph{Mathematica}) in the following form:
\begin{corollary}\label{cor}
For all integers $n\ge 158$ and for all $\delta$ such that
\begin{equation*}
0.227\le \delta\le 2.9999,
\end{equation*}
we have
$$c_3^{(\delta)}(n)\,c_3^{(\delta)}(n+3)>0.$$
\end{corollary}

\begin{remark}
While Corollary~\ref{cor} only partially affirms Conjecture~\ref{conj1},
it also gives information about the cases when
$0.227\le\delta<\frac{9-\sqrt{73}}{2}$ and $1<\delta<2$
(not covered by the conjecture).
In these cases case the corollary tells us that the
Borwein coefficients $c_3^{(\delta)}(n)$ satisfy the respective
sign pattern for large enough $n$ (namely $n\ge 158$).
\end{remark}

For the exponent $\delta=3$ we actually have the following
result for any prime $p$ which is a cubic analogue of Theorem~\ref{thm:a}:
\begin{theorem}\label{thm:a3}
For all primes $p$, $c_p^{(3)}(n)$ and $c_p^{(3)}(n+p)$
have the same sign for each $n\ge 0$, i.e.,
$$c_{p}^{(3)}(n)c_{p}^{(3)}(n+p)\ge 0,$$
for each $n\ge 0$.
\end{theorem}
The proof is given in Section~\ref{sec:van}.

Our paper is organized as follows: In Section~\ref{sec:proof} we prove
Theorem \ref{mth}, thus establish an asymptotic formula for the Borwein
coefficients $c_p^{(\delta)}(n)$, for any prime $p$.
In Section~\ref{sec:sp} we turn to the $p=3$ case.
We prove Theorem~\ref{mth1} there, which provides us with
a useful estimate for the growth
of the coefficients $c_3^{(\delta)}(n)$.
This allows us to prove Corollary~\ref{cor}.
In Section~\ref{sec:van}, which is of independent interest,
we prove some results that include vanishing and divisibility properties
for the Borwein coefficients of the cube of the infinite Borwein product.
Finally, in Appendix~\ref{sec:app} we present several new conjectures
on precise sign patterns of infinite Borwein products
and other products raised to a real power, which are similar to
Conjecture~\ref{conj1}.

\section{The proof of Theorem \ref{mth}}\label{sec:proof}

Our proof is in two steps.
In the first step we establish a modular transformation for the generating
function $G_{p}(q)^{\delta}$. In the next step we follow Rademacher's
method and use the modular transformation to obtain an expansion
for the Borwein coefficients $c_p^{(\delta)}(n)$.

Recall that $f(q)$ and $G_p(q)$ were defined in
\eqref{eq:f} and \eqref{eq:Gp}, respectively.
Further, we would like to remind the reader about the notation
$\omega_{h,k}$ used for certain exponentials involving Dedekind sums,
see \eqref{om} and \eqref{Dedsum},
which prominently appear in the
Hardy--Ramanujan--Rademacher circle method (see \eqref{part-asym}).

\subsection{Modular transformation for the generating
  function}\label{ssec:mt}
\begin{proposition}\label{prop1}
Let $h,k\in\zb$ such that $k>0$ and $\gcd(h,k)=1$.
Let $d=\gcd(p,k)$, and let $h'$ and $h_d'$ be solutions of the congruences
\begin{align*}
hh'&\equiv 1 ~(\bmod~k)\\
\intertext{and}
(hp/d)h_d'&\equiv 1 ~(\bmod ~{k/d}).
\end{align*}
Then, for all $\delta\in\rb$ and $\Re(z)>0$ we have
\begin{align*}
  G_{p}\Big(e^{\frac{2\pi\ri h}{k}-\frac{2\pi z}{k^2}}\Big)^{\delta}
  =&\Big(\frac{p}{d}\Big)^{\frac{\delta}{2}}
     \Big(\omega_{h,k}^{-1}\omega_{\frac{ph}{d},\frac{k}{d}}\Big)^{\delta}
     \exp\bigg(\frac{\pi\delta(d^2-3)}{36z}-\frac{\pi \delta z}{6k^2}\bigg)
     \hat{G}_{p}\Big(h,k; e^{-\frac{2\pi}{z}}\Big)^{\delta},
\end{align*}
where
$$
\hat{G}_{p}\Big(h,k; e^{-\frac{2\pi}{z}}\Big)
=f\bigg(e^{\frac{2\pi\ri dh_d'}{k}-\frac{2\pi d^2}{pz}}\bigg)
f\Big(e^{\frac{2\pi\ri h'}{k}-\frac{2\pi}{z}}\Big)^{-1}.
$$
\end{proposition}

\begin{proof}
 Notice that the functions occurring in the statement of the proposition
 are all holomorphic on $\Re(z)>0$.
 We just need to show that the transformation holds for all positive real $z$,
then the full transformation follows by analytic continuation.
For $z>0$, from Hardy and Ramanujan \cite[Lemma~4.31]{MR1575586},
we have
\begin{equation*}
  f\Big(e^{\frac{2\pi\ri h}{k}-\frac{2\pi z}{k^2}}\Big)
  =\omega_{h,k}\Big(\frac{z}{k}\Big)^{\frac{1}{2}}
  \exp\bigg(\frac{\pi}{12z}-\frac{\pi z}{12k^2}\bigg)
  f\bigg(e^{\frac{2\pi\ri h'}{k}-\frac{2\pi}{z}}\bigg),
\end{equation*}
where $\omega_{h,k}=e^{\pi\ri s(h,k)}$, and $s(h,k)$ is defined
by \eqref{Dedsum}. Taking into account $d=\gcd(p,k)$, we thus have
\begin{align*}
  G_{p}\Big(e^{\frac{2\pi\ri h}{k}-\frac{2\pi z}{k^2}}\Big)
  &=f\bigg(e^{\frac{2\pi\ri h(p/d)}{k/d}-\frac{2\pi (pz/d^2)}{(k/d)^2}}\bigg)
    f\Big(e^{\frac{2\pi\ri h}{k}-\frac{2\pi z}{k^2}}\Big)^{-1}\\
  &=\omega_{\frac{hp}{d},\frac{k}{d}}\bigg(\frac{pz/d^2}{k/d}\bigg)^{\frac{1}{2}}
    \exp\bigg(\frac{\pi}{12(pz/d^2)}-\frac{\pi (pz/d^2)}{12(k/d)^2}\bigg)
    f\bigg(e^{\frac{2\pi\ri h_d'}{(k/d)}-\frac{2\pi}{pz/d^2}}\bigg)\\*
  &\quad\;\times \omega_{h,k}^{-1}\Big(\frac{z}{k}\Big)^{-\frac{1}{2}}
    \exp\bigg(-\frac{\pi}{12z}+\frac{\pi z}{12k^2}\bigg)
    f\bigg(e^{\frac{2\pi\ri h'}{k}-\frac{2\pi}{z}}\bigg)^{-1}\\
  &=\omega_{h,k}^{-1}\omega_{\frac{hp}{d},\frac{k}{d}}
    \Big(\frac{p}{d}\Big)^{\frac{1}{2}}
    \exp\bigg(\frac{\pi(d^2-p)}{12pz}-\frac{(p-1)\pi z}{12k^2}\bigg)
    \hat{G}_{p}\Big(h,k; e^{-\frac{2\pi}{z}}\Big).
\end{align*}
From this we obtain for all $\delta>0$,
\begin{align*}
  G_{p}\Big(e^{\frac{2\pi\ri h}{k}-\frac{2\pi z}{k^2}}\Big)^{\delta}
  &=\Big(\omega_{h,k}^{-1}\omega_{\frac{hp}{d},\frac{k}{d}}
    \hat{G}_{p}\Big(h,k; e^{-\frac{2\pi}{z}}\Big)\Big)^{\delta}\\*
  &\quad\;\times\Big(\frac{p}{d}\Big)^{\frac{\delta}{2}}
    \exp\bigg(\frac{\delta(d^2-p)\pi}{12pz}
    -\frac{\delta(p-1)\pi z}{12k^2}\bigg).
\end{align*}
It is obvious that
$G_{p}\Big(e^{\frac{2\pi\ri h}{k}-\frac{2\pi z}{k^2}}\Big)^\delta$
can be analytically extended to a single valued analytic function
on the right half plane $\Re(z)>0$. So the function on the right-hand
side of the equation above has the same properties.
Furthermore, one can find a small open set $\Omega$ on the right
half plane $\Re(z)>0$ such that
$$\arg\Big(\omega_{h,k}^{-1}\omega_{\frac{hp}{d},\frac{k}{d}} \Big)
+ \arg\Big(\hat{G}_{p}\Big(h,k; e^{-\frac{2\pi}{z}}\Big)\Big) \in [-\pi,\pi],$$
for all $z\in \Omega$.  This implies that if $z\in \Omega$ then
\begin{align*}
  G_{p}\Big(e^{\frac{2\pi\ri h}{k}-\frac{2\pi z}{k^2}}\Big)^{\delta}
  &=\Big(\omega_{h,k}^{-1}\omega_{\frac{hp}{d},\frac{k}{d}}\Big)^{\delta}
    \hat{G}_{p}\Big(h,k; e^{-\frac{2\pi}{z}}\Big)^{\delta}\\*
  &\quad\;\times\Big(\frac{p}{d}\Big)^{\frac{\delta}{2}}
    \exp\bigg(\frac{\delta(d^2-p)\pi}{12pz}-
    \frac{\delta(p-1)\pi z}{12k^2}\bigg).
\end{align*}
Finally, noticing that $\hat{G}_{p}\Big(h,k; e^{-\frac{2\pi}{z}}\Big)^{\delta}$
can be analytically extended to a single valued analytic function
on the right half plane $\Re(z)>0$, the proof of the proposition is complete.
\end{proof}

\subsection{Rademacher expansion for the Borwein
  coefficients}\label{ssec:re}
Let $n,N\in\nb$. Following Rademacher \cite{MR8618} or
\cite[Equation~(117.1)]{MR0364103},
we have
\begin{equation}\label{eqm0}
  c_p^{(\delta)}(n)=\sum_{1\le k\le N}\sum_{\substack{0\le h<k\\ \gcd(h,k)=1}}
  \frac{\ri}{k^2}e^{-\frac{2\pi\ri hn}{k}}\int_{z_{h,k}'}^{z_{h,k}''}
  G_{p}\Big(e^{\frac{2\pi\ri h}{k}-\frac{2\pi z}{k^2}}\Big)^{\delta}
  \exp\bigg(\frac{2\pi nz}{k^2}\bigg)\rd z,
\end{equation}
where $z$ runs in each integral on an arc of the circle
$K: \left|z-{1}/{2}\right|={1}/{2}$ with $\Re(z)>0$,
with the ends $z_{h,k}'$ and $z_{h,k}''$
(see the segment in blue in Figure \ref{f1}) of the arc being given by
\begin{equation}\label{eqzhk}
  z_{h,k}'=\frac{k^2}{k^2+k_1^2}+\ri\frac{kk_1}{k^2+k_1^2}
  \quad\mbox{and}\quad
  z_{h,k}''=\frac{k^2}{k^2+k_2^2}-\ri\frac{kk_2}{k^2+k_2^2},
\end{equation}
respectively. Here $k_1,k_2\in\nb$ are taken from the denominators of
adjoint points of $h/k$ in the Farey sequence of order $N$.

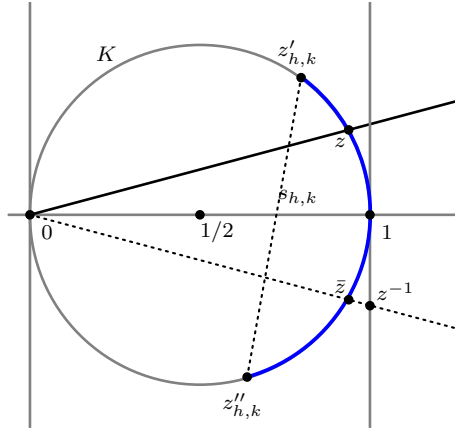
\begin{figure}
\centering
    \label{zplane}
    \begin{tikzpicture}[line cap=round,line join=round,x=1cm,y=1cm, thick,scale=0.565]
   \clip(-2,-5) rectangle (10,5);
\draw [line width=1pt,color=gray] (-0.5,0)-- (12,0);
\draw [line width=1pt,color=gray] (0,-5)-- (0,5);
\draw [line width=1pt,color=black] (0,0)-- (12,3.2);
\draw [dotted, color=black] (0,0)-- (12,-3.2);
\draw [line width=1pt,color=gray] (8,-5)-- (8,5);
\draw [dotted, color=black]  (5.11,-3.82)-- (6.38,3.23);
\draw [line width=1pt,color=gray] (4,0) circle (4 cm);
\draw [line width=1.4pt, color=blue] (5.11,-3.82) arc (-74:54:4);
\begin{scriptsize}
\draw [fill=black] (0,0) circle (2.5pt);
\draw[color=black] (0.4,-0.4) node {$0$};
\draw [fill=black] (4,0) circle (2.5pt);
\draw [fill=black] (8,0) circle (2.5pt);
\draw[color=black] (4.4,-0.4) node {${1}/{2}$};
\draw[color=black] (8.4,-0.4) node {$1$};
\draw [fill=black] (6.38,3.23) circle (2.5pt);
\draw[color=black] (6.3,3.8) node {$z_{h, k}'$};
\draw [fill=black] (5.11,-3.82) circle (2.5pt);
\draw[color=black] (5,-4.5) node {$z_{h, k}''$};
\draw[color=black] (6.3,0.4) node {$s_{h, k}$};
\draw [fill=black] (7.5,2) circle (2.5pt);
\draw [fill=black] (7.5,-2) circle (2.5pt);
\draw[color=black] (7.3,1.7) node {$z$};
\draw[color=black] (7.3,-1.7) node {$\bar{z}$};
\draw [fill=black] (8,-2.14) circle (2.5pt);
\draw[color=black] (8.6,-1.8) node {$z^{-1}$};
\draw[color=black] (1.8,3.8) node {$K$};
\end{scriptsize}
\end{tikzpicture}
\caption{Path of integration in the $z$-plane}\label{f1}

\end{figure}

Applying Proposition~\ref{prop1} to Equation~\eqref{eqm0} yields
\begin{align*}
  c_p^{(\delta)}(n)
  &=\Bigg(\sum_{\substack{1\le k\le N\\ p\mid k}}+
  \sum_{\substack{1\le k\le N\\ p\nmid k}}\Bigg)
  \sum_{\substack{0\le h<k\\ \gcd(h,k)=1}}
  \frac{\ri \big(\frac pd\big)^{\frac{\delta}{2}}
  \Big(\omega_{h,k}^{-1}\omega_{\frac{hp}{d},\frac{k}{d}}\Big)^{\delta}
  e^{-\frac{2\pi\ri hn}{k}}}{k^2}\\*
  &\quad\;\times\int_{z_{h,k}'}^{z_{h,k}''}
    e^{\frac{\pi\delta(d^2-p)}{12pz}+\frac{(24n-(p-1)\delta)\pi z}{12k^2}}
    \hat{G}_{p}\Big(h,k; e^{-\frac{2\pi}{z}}\Big)^{\delta}\rd z\\
&=:I+E.
\end{align*}
Notice that for the second sum $E$, the relation $p\nmid k$ means
$d=\gcd(p,k)=1$; we obtain
\begin{align*}
  E=\sum_{\substack{1\le k\le N\\ p\nmid k}}\sum_{\substack{0\le h<k\\ \gcd(h,k)=1}}
  &\frac{\ri\,p^{\frac{\delta}{2}}
    \Big(\omega_{h,k}^{-1}\omega_{ph,k}\Big)^{\delta}
    e^{-\frac{2\pi\ri hn}{k}}}{k^2}\\*
  &\times\int_{z_{h,k}'}^{z_{h,k}''}
    e^{-\frac{\pi\delta(p-1)}{12pz}+\frac{(24n-(p-1)\delta)\pi z}{12k^2}}
    \hat{G}_{p}\Big(h,k; e^{-\frac{2\pi}{z}}\Big)^{\delta}\rd z.
\end{align*}
The path of integration in above inner sum, which is an arc of the circle $K$,
can here be replaced by the chord $s_{h,k}$ from $z_{h,k}'$ to $z_{h,k}''$.
On the chord $s_{h,k}$, from
Rademacher~\cite[Equations~(119.3) and (119.6)]{MR0364103},
we have
$$1\le \Re\bigg(\frac{1}{z}\bigg),\quad 0< \Re(z)<\frac{2k^2}{N^2},$$
and, from Rademacher~\cite[Equation~(3.5)]{MR8618},
the length of the chord $s_{h,k}$ is
$$|s_{h,k}|<\frac{2k}{N+1}.$$
Thus we deduce
\begin{align*}
  |E|&\le \sum_{\substack{1\le k\le N\\ p\nmid k}}
  \sum_{\substack{0\le h<k\\ \gcd(h,k)=1}}
  \frac{ p^{\frac{\delta}{2}}|s_{h,k}|}{k^2}\sup_{z\in s_{h,k}}
  \bigg|e^{-\frac{\pi\delta(p-1)}{12pz}+\frac{(24n-(p-1)\delta)\pi z}{12k^2}}
  \hat{G}_{p}\Big(h,k; e^{-\frac{2\pi}{z}}\Big)^{\delta}\bigg|\\
     &\le \sum_{\substack{1\le k\le N\\ p\nmid k}}
  \sum_{\substack{0\le h<k\\ \gcd(h,k)=1}}\frac{ 2\cdot p^{\frac{\delta}{2}}
  e^{-\frac{\pi\delta(p-1)}{12p}+\frac{(24n-(p-1)\delta)\pi }{6N^2}}}{k(N+1)}
  \sup_{z\in s_{h,k}}
  \left|\frac{f\bigg(e^{\frac{2\pi\ri h_1'}{k}-\frac{2\pi }{pz}}\bigg)}
  {f\Big(e^{\frac{2\pi\ri h'}{k}-\frac{2\pi}{z}}\Big)}\right|^{\delta}.
\end{align*}
Further, it is not difficult to see that
$$\sup_{z\in s_{h,k}}\left|f\bigg(e^{\frac{2\pi\ri h_1'}{k}-\frac{2\pi }{pz}}\bigg)
  f\Big(e^{\frac{2\pi\ri h'}{k}-\frac{2\pi}{z}}\Big)^{-1}\right|\le
f\Big(e^{-\frac{2\pi }{p}}\Big)f\Big(e^{-2\pi}\Big),$$
and hence
\begin{align}\label{eqE}
|E|
  \le  \sum_{\substack{1\le k\le N\\ p\nmid k}}
  \sum_{\substack{0\le h<k\\ \gcd(h,k)=1}}\frac{ 2\cdot p^{\frac{\delta}{2}}
  e^{-\frac{\pi\delta(p-1)}{12p}+\frac{(24n-(p-1)\delta)\pi }{6N^2}}}{k(N+1)}
  f\Big(e^{-\frac{2\pi }{p}}\Big)f\Big(e^{-2\pi}\Big).
\end{align}

For the first sum $I$, the relation $p\mid k$ means $d=\gcd(p,k)=p$;
we now obtain the estimate
\begin{align*}
  I&=\sum_{\substack{1\le k\le N\\ p\mid k}}
  \sum_{\substack{0\le h<k\\ \gcd(h,k)=1}}
  \frac{\ri \Big(\omega_{h,k}^{-1}\omega_{h,\frac{k}{p}}\Big)^{\delta}
  e^{-\frac{2\pi\ri hn}{k}}}{k^2}\\*
  &\qquad\times \int_{z_{h,k}'}^{z_{h,k}''}
    e^{\frac{\pi(p-1)\delta}{12z}+\frac{(24n-(p-1)\delta)\pi z}{12k^2}}
    \bigg(1+\Big(\hat{G}_{p}
    \Big(h,k; e^{-\frac{2\pi}{z}}\Big)^{\delta}-1\Big)\bigg)
    \rd z\\
&=:I_M+I_R,
\end{align*}
where the division of $I$ into the two terms $I_M$ and $I_R$ comes
from splitting the last factor in the integrand into $1$ and
$\Big(\hat{G}_{p}\Big(h,k; e^{-\frac{2\pi}{z}}\Big)^{\delta}-1\Big)$.
Similar to the estimate for $E$, for $I_R$ we have
\begin{align*}
  |I_R|\le \sum_{\substack{1\le k\le N\\ p\mid k}}
  &\qquad\sum_{\substack{0\le h<k\\ \gcd(h,k)=1}}
  \frac{2e^{\frac{(24n-(p-1)\delta)\pi z}{6N^2}}}{k(N+1)}
  \sup_{z\in s_{h,k}}\bigg|e^{\frac{\pi(p-1)\delta}{12z}}
  \Big(\hat{G}_{p}\Big(h,k; e^{-\frac{2\pi}{z}}\Big)^{\delta}-1\Big)\bigg|.
\end{align*}
For all $t_1,t_2\in\cb$ with $|t_1|, |t_2|<1$ and $\delta>0$,
it not difficult to show that
$$\bigg|\Big(f(t_1)f(t_2)^{-1}\Big)^{\delta}-1\bigg|
\le f(|t_1|)^{\delta}f(|t_2|)^{\delta}-1,$$
by using the definition of $f(t)$. Hence we obtain
\begin{align*}
  &\sup_{z\in s_{h,k}}
  \bigg|e^{\frac{\pi\delta(p-1)}{12pz}}
  \Big(\hat{G}_{p,\delta}\Big(h,k; e^{-\frac{2\pi}{z}}\Big)-1\Big)\bigg|\\
  &\le \sup_{z\in s_{h,k}}\bigg|e^{\frac{\pi\delta(p-1)}{12}\Re(z^{-1})}
    \Big(f\Big(e^{-2p\pi \Re(z^{-1})}\Big)^{\delta}
    f\Big(e^{-2\pi\Re(z^{-1})}\Big)^{\delta}-1\Big)\bigg|\\
  &\le e^{\frac{\pi\delta(p-1)}{12}}\bigg(f\Big(e^{-2p\pi }\Big)^{\delta}
    f\Big(e^{-2\pi}\Big)^{\delta}-1\bigg),
\end{align*}
by using $\delta\in(0,24/(p-1)]$
and the definition of $\hat{G}_{p}(h,k; e^{-\frac{2\pi}{z}})$. Therefore
\begin{align}\label{eqIR}
  |I_R|\le \sum_{\substack{1\le k\le N\\ p\mid k}}
  \sum_{\substack{0\le h<k\\ \gcd(h,k)=1}}
  \frac{ 2e^{\frac{\pi\delta(p-1)}{12}+\frac{(24n-(p-1)\delta)\pi }{6N^2}}}{k(N+1)}
  \bigg(f\Big(e^{-2p\pi }\Big)^{\delta}
  f\Big(e^{-2\pi}\Big)^{\delta}-1\bigg).
\end{align}

To evaluate $I_M$ we split integral into two parts, $I_{MM}$ and
$I_{ME}$, as indicated below; the path of integration
of the first part is the whole circle $K$, traversed
from $0$ to $0$ in negative direction, while in the second part
the arc of the circle is traversed from $z_{h,k}'$ to $z_{h,k}''$
(which itself is split into a difference of two integrals
with respective paths of integration along the arc starting in $0$).
Specifically, we have
\begin{align*}
  I_M&=\sum_{\substack{1\le k\le N\\ p\mid k}}
  \sum_{\substack{0\le h<k\\ \gcd(h,k)=1}}
  \frac{\ri \Big(\omega_{h,k}^{-1}\omega_{h,\frac{k}{p}}\Big)^{\delta}
  e^{-\frac{2\pi\ri hn}{k}}}{k^2}\\*
     &\qquad\times
       \Bigg(\int_{K}-\Bigg(\int_{0}^{z_{h,k}'}-\int_{0}^{z_{h,k}''}\Bigg)\Bigg)
       e^{\frac{\pi\delta(p-1)}{12z}+\frac{(24n-(p-1)\delta)\pi z}{12k^2}}\rd z\\
       &=:I_{MM}+I_{ME}.
\end{align*}

We estimate the second part, $I_{ME}$, first.
On the arc from $0$ to $z_{h,k}'$,
and the arc from $0$ to $z_{h,k}''$, from
Rademacher~\cite[Equation~(119.6)]{MR0364103}
and  \cite[Equation~(120.2)]{MR0364103}, we have
 $$0\le \Re(z)\le \max(\Re(z_{h,k}'),\Re(z_{h,k}'')) <\frac{2k^2}{N^2},$$
 and $\Re(1/z)=1$.
 The lengths of both, the arc from $0$ to $z_{h,k}'$,
 and the arc from $0$ to $z_{h,k}''$, are less than
 $$\frac{\pi}{2}|z_{h,k}'|, \;\frac{\pi}{2}|z_{h,k}''|
 \le \frac{\pi k}{\sqrt{2}N}.$$
Together, we obtain
\begin{equation}\label{eqIME}
  |I_{ME}|\le 2\sum_{\substack{1\le k\le N\\ p\mid k}}
  \sum_{\substack{0\le h<k\\ \gcd(h,k)=1}}
  \frac{1}{k^2}\frac{\pi k}{\sqrt{2}N}
  e^{\frac{\pi\delta(p-1)}{12}+\frac{(24n-(p-1)\delta)\pi }{6N^2}}.
\end{equation}

Finally, for the first part
which is
$$I_{MM}=\sum_{\substack{1\le k\le N\\ p\mid k}}
\sum_{\substack{0\le h<k\\ \gcd(h,k)=1}}
\frac{\ri \Big(\omega_{h,k}^{-1}\omega_{h,\frac{k}{p}}\Big)^{\delta}
  e^{-\frac{2\pi\ri hn}{k}}}{k^2}\int_{K}
e^{\frac{\pi\delta(p-1)}{12z}+\frac{(24n-(p-1)\delta)\pi z}{12k^2}}\rd z,$$
we use the integral representation for the modified
Bessel function $I_1$ in \eqref{eq:intrpI1},
and apply the substitution $w=1/z$.
Simplification then gives
\begin{equation}\label{eqIMM}
  I_{MM}=\frac{2\pi \delta^{1/2}}{\sqrt{\frac{24n}{p-1}-\delta}}
  \sum_{\substack{1\le k\le N\\ p\mid k}}A_{k,\delta}(n)
  I_1\Bigg(\frac{(p-1)\pi}{6k}
    \sqrt{\delta\bigg(\frac{24n}{p-1}-\delta\bigg)}\Bigg),
\end{equation}
where
$$A_{k,\delta}(n)=\frac{1}{k}\sum_{\substack{0\le h<k\\ \gcd(h,k)=1}}
\Big(\omega_{h,k}^{-1}\omega_{h,\frac{k}{p}}\Big)^{\delta}
e^{-\frac{2\pi\ri hn}{k}}.$$
From Equations \eqref{eqE}, \eqref{eqIR}, \eqref{eqIME} and \eqref{eqIMM},
we therefore obtain
\begin{align*}
&\big|c_p^{(\delta)}(n)-I_{MM}\big|\\
&\le |I_{ME}|+|I_{R}|+|E|\\
  &\le  \frac{e^{\frac{(p-1)\pi\delta}{12}+\frac{(24n-(p-1)\delta)\pi }{6N^2}} }{N}
    \sum_{\substack{1\le k\le N\\ p\mid k}}
  \frac{\varphi(k)}{k}\bigg(\pi\sqrt{2}
  +2\bigg(f\Big(e^{-2p\pi }\Big)^{\delta}
  f\Big(e^{-2\pi}\Big)^{\delta}-1\bigg)\bigg)\\*
  &\quad\;+ \sum_{\substack{1\le k\le N\\ 3\nmid k}}
  \sum_{\substack{0\le h<k\\ \gcd(h,k)=1}}
  \frac{ 2\cdot p^{\frac{\delta}{2}}
  e^{-\frac{\pi(p-1)\delta}{12p}+\frac{(24n-(p-1)\delta)\pi }{6N^2}}}{k(N+1)}
  f\Big(e^{-\frac{2\pi }{p}}\Big)^{\delta}f\Big(e^{-2\pi}\Big)^{\delta},
\end{align*}
where $\varphi(k)$ is Euler's totient function.
Notice that if $p\mid k$ then $\varphi(k)\le (1-1/p)k$.
Substituting $N\mapsto pN$, the error term becomes
\begin{align*}
  &\big|c_p^{(\delta)}(n)-I_{MM}\big|\\
  &\le \frac{(p-1)e^{\frac{(p-1)\pi\delta}{12}+
    \frac{(24n-(p-1)\delta)\pi }{6(pN)^2}} }{p^2}
    \bigg(\pi\sqrt{2}+2\bigg(f\Big(e^{-6\pi }\Big)^{\delta}
    f\Big(e^{-2\pi}\Big)^{\delta}-1\bigg)\bigg)\\*
  &\quad\;+ \frac{ 2(p-1)\cdot
    e^{-\frac{\pi(p-1)\delta}{12p}+\frac{(24n-(p-1)\delta)\pi }{6(pN)^2}}}
    {p^{1-\delta/2}}f\Big(e^{-\frac{2\pi }{p}}\Big)^{\delta}
    f\Big(e^{-2\pi}\Big)^{\delta}\\
  &= e^{\frac{(24n-(p-1)\delta)\pi }{6p^2N^2}}
    \Bigg(\frac{(p-1)e^{\frac{(p-1)\pi\delta}{12}} }{p^2}
    \bigg(\pi\sqrt{2}-2+2f\Big(e^{-6\pi }\Big)^{\delta}
    f\Big(e^{-2\pi}\Big)^{\delta}\bigg)\\*
  &\quad\; + \frac{ 2(p-1)\cdot e^{-\frac{\pi(p-1)\delta}{12p}}}{p^{1-\delta/2}}
    f\Big(e^{-\frac{2\pi }{p}}\Big)^{\delta}
    f\Big(e^{-2\pi}\Big)^{\delta}\Bigg).
\end{align*}
This completes the proof of Theorem~\ref{mth}.\qed

\section{Sign pattern for the Borwein coefficients with
  $p=3$}\label{sec:sp}
\subsection{The proof of Theorem \ref{mth1}}
We now assume that $\delta\in(0, 3]$, and start with $n\ge 1$
and will later restrict to $n\ge 158$.
We let
\begin{equation}\label{eqld}
L_{\delta,n}=\frac{\pi}{18}\sqrt{\delta(12n-\delta)},\quad
\hat{c}_{\delta}(n)=\frac{2\pi^2\delta}{27} L_{\delta,n}^{-1}\,
I_1\big(2L_{\delta,n}\big)
\end{equation}
(remember that  $I_1$ is the modified Bessel function,
defined in \eqref{eqbs}),
and fix
$$N= \Big\lceil\big(20 L_{\delta,n}^2/\delta\big)^{1/2}\Big\rceil.$$

Substituting these into Theorem~\ref{mth} and taking $p=3$,
using computer algebra (we utilized \emph{Mathematica} and found it
convenient to rewrite the occurrences of $f(\cdot)$ using
$f(e^{-2\pi x})=e^{2\pi x/24}\eta(\ri x)^{-1}$ where
$\eta(x)$ is the classical Dedekind eta function,
already implemented as a built-in function in \emph{Mathematica}),
we find that
\begin{align}\label{eqedl}
 \Big|E_{3,N}^{(\delta)}(n)\Big|\le \frac{4}{9}\cdot e^{\frac{3}{5\pi}}
  \Big(1.689^{\delta}\big(1.222+1.002^{\delta}\big)+3\cdot 1.692^{\delta}\Big)
\end{align}
for any integer $n\ge 1$.
Furthermore, by noticing that $\varphi(3k)\le 2k$, we have
\begin{align}\label{eqcd}
  &\bigg|c_{\delta}(n)-\hat{c}_{\delta}(n)
    \cos\bigg(\frac{\pi\delta}{18}+\frac{2\pi n}{3}\bigg)\bigg|\nonumber\\
  &\qquad\qquad\le \Big|E_{3,N}^{(\delta)}(n)\Big|+
    \frac{2\pi \delta^{1/2}}{\sqrt{12n-\delta}}
    \sum_{2\le k\le N}\frac{\varphi(3k)}{3k}\,
    I_1\Big(\frac{\pi}{9k}\sqrt{\delta(12n-\delta)}\Big)\nonumber\\
  &\qquad\qquad\le \Big|E_{3,N}^{(\delta)}(n)\Big|+
  \frac{2\pi^2\delta}{27L_{\delta, n}}
    \sum_{2\le k\le N}I_1\bigg(\frac{2}{k}L_{\delta,n}\bigg).
\end{align}
To give an upper bound for the above sum we require the following lemma.
\begin{lemma}\label{lem:e}
  For any real $x>0$ and integer $y>2$ we have
  $$\sum_{2\le k\le y}I_1\bigg(\frac{2x}{k}\bigg)\le
  x\log y+2I_1\left(x\right)-\bigg(2-\gamma-\frac 1{2y}\bigg)x.$$
Here $\gamma=0.577216\ldots$ is the Euler--Mascheroni constant.
\end{lemma}
\begin{proof}Using the well-known bound
$$\sum_{1\le k\le y}\frac{1}{k}\le \log y+\gamma+\frac{1}{2y},$$
we find that
\begin{align*}
  \sum_{2\le k\le y}I_1\left(\frac{2x}{k}\right)
  &=\sum_{n\ge 0}\frac{x^{2n+1}}{n!(n+1)!}\sum_{2\le k\le y}\frac{1}{k^{2n+1}}\\
  &=x\sum_{2\le k\le y}\frac{1}{k}+
    \sum_{n\ge 1}\frac{x^{2n+1}}{n!(n+1)!}\sum_{2\le k\le y}\frac{1}{k^{2n+1}}\\
  &\le x\bigg(\log y+\gamma-1+\frac{1}{2y}\bigg)+
    2\Bigg(\sum_{n\ge 0}\frac{(x/2)^{2n+1}}{n!(n+1)!}-\frac{x}{2}\Bigg)\\
&\le x\log y+2I_1\left(x\right)-\bigg(2-\gamma-\frac 1{2y}\bigg)x,
\end{align*}
which completes the proof.
\end{proof}
Now, applying Lemma~\ref{lem:e} to \eqref{eqcd}, we have
\begin{align*}
  &\bigg|c_{\delta}(n)-\hat{c}_{\delta}(n) \cos\bigg(\frac{\pi\delta}{18}
    +\frac{2\pi n}{3}\bigg)\bigg|\\
  &\le \Big|E_{3,N}^{(\delta)}(n)\Big|+\frac{2\pi^2\delta}{27L_{\delta, n}}
    \big(L_{\delta,n}\log N+2I_1(L_{\delta,n})
    -(2-\gamma-(2N)^{-1})L_{\delta,n}\big).
\end{align*}
Now we let $n\ge 158$. Since
$$
20 L_{\delta,n}^2/\delta=\frac{20\pi^2(12n-\delta)}{18^2}\ge
\frac{20\pi^2\cdot (12\cdot 158-3)}{18^2}\ge 1153,
$$
we have
$$
N=\Big\lceil(20 L_{\delta,n}^2/\delta)^{1/2}\Big\rceil\le
\big(1+1153^{-1/2}\big)\big(20 L_{\delta,n}^2/\delta\big)^{1/2}.
$$
If we also insert the definition of $\hat{c}_{\delta}(n)$ we obtain that
\begin{align*}
  &\bigg|\frac{c_{\delta}(n)}{\hat{c}_{\delta}(n)}-
    \cos\bigg(\frac{\pi\delta}{18}+\frac{2\pi n}{3}\bigg)\bigg|\\
  &\le \frac{L_{\delta,n}\log L_{\delta,n}+2 I_1(L_{\delta,n})}
    {I_1(2L_{\delta,n})}\\*
  &\quad\;+\frac{\Big|E_{3,N}^{(\delta)}(n)\Big|+\frac{2\pi^2\delta}{27}
    \Big(\log \big(\big(1+1153^{-1/2}\big)(20/\delta)^{1/2}\big)
    -\big(2-\gamma-1153^{-1/2}/2\big)\Big)}
    {\frac{2\pi^2\delta}{27} L_{\delta,n}^{-1}\,I_1(2L_{\delta,n})}.
\end{align*}
Using \emph{Mathematica} and inserting \eqref{eqedl}, we obtain
\begin{align*}
  \left|\frac{c_{\delta}(n)}{\hat{c}_{\delta}(n)}-
  \cos\left(\frac{\pi\delta}{18}+\frac{2\pi n}{3}\right)\right|
  \le \frac{L_{\delta,n} w(\delta)+L_{\delta,n}\log L_{\delta,n}+
  2 I_1(L_{\delta,n})}{I_1\left(2L_{\delta,n}\right)},
\end{align*}
with
\begin{equation}\label{eqwd}
w(\delta)=\frac{1}{2}\log\bigg(\frac{1}{\delta}\bigg)+
\frac{0.736\,\big(1.689^{\delta}\big(1.222+1.002^{\delta}\big)+
    3\cdot 1.692^{\delta}\big)}{\delta}+0.119.
\end{equation}
This completes the proof of Theorem \ref{mth1}.\qed

\subsection{The proof of Corollary \ref{cor}}
To establish Corollary \ref{cor}, we will make use of the following lemma.
\begin{lemma}\label{lem32}Let $L_{\delta,n}$ and $w(\delta)$ be given as
  in \eqref{eqld} and \eqref{eqwd}, respectively.
  For each $\delta \in (0, 3]$, define
  $$M(L_{\delta,n}):=\frac{L_{\delta,n} w(\delta)+L_{\delta,n}\log L_{\delta,n}
    +2 I_1(L_{\delta,n})}{I_1\left(2L_{\delta,n}\right)}.$$
  Then $M(L_{\delta,n})$ is decreasing in $n$ for $n\ge 158$,
  whenever $\delta \in [0.227, 3]$.
\end{lemma}
\begin{proof}We have $L_{\delta,n}=(\pi/18)\sqrt{\delta(12n-\delta)}\ge 3.6$
  for all $\delta\in[0.227,3]$ and $n\ge 158$, and $L_{\delta,n}$ is increasing
  for all $n\ge 1$. Hence we just need to prove that
 $$M(u):=\frac{u w(\delta)+u\log u+2 I_1(u)}{I_1\left(2u\right)},$$
 is decreasing for all $u\ge 3$. We have $w(\delta)>0$ for all
 $\delta\in [0.227, 3]$.
 Using the definition of $I_1(u)$ in \eqref{eqbs}, it is clear that
 $uw(\delta)/I_1(2u)$ is decreasing for $u>0$. Also,
\begin{align*}
  \bigg(\frac{u\log u}{I_1(2u)}\bigg)'=
  \frac{u^{-1}(u^{-1}I_1(2u))-(u^{-1}I_1(2u))'\log u}{(u^{-1}I_1(2u))^2},
\end{align*}
and
\begin{align*}
  u^{-1}(u^{-1}I_1(2u))-(u^{-1}I_1(2u))'\log u
  &=\sum_{\ell\ge 0}\frac{u^{2\ell-1}}{\ell !(\ell+1)!}-
    \log u \sum_{\ell\ge 0}\frac{2\ell u^{2\ell-1}}{\ell !(\ell+1)!}\\
  &=\frac{1}{u}-
    \sum_{\ell\ge 1}\frac{(2\ell\log u-1) u^{2\ell-1}}{\ell !(\ell+1)!}\\
&\le \frac{1}{u}-\frac{u}{2}(2\log u-1),
\end{align*}
for all $u\ge \sqrt{e}$.
Clearly, $1/u-u(2\log u-1)/2$ is decreasing and not greater than
$$\frac{1}{3}-\frac{3}{2}(2\log 3-1)<0,$$
for $u\ge 3$. This means that ${u\log u}/{I_1(2u)}$ is decreasing for $u\ge 3$.

The more difficult part is to prove that $I_1(u)/I_1(2u)$ is
decreasing for $u\ge 3$.
We shall prove
\begin{align}\label{eqbin}
  \bigg(\frac{I_1(u)}{I_1(2u)}\bigg)'
  =\frac{I_1'(u)I_1(2u)-2I_1(u)I_1'(2u)}{I_1(2u)^2}\le 0,
\end{align}
for all $u>0$. Inserting the well-known functional relation for the
modified {B}essel function $I_1$, namely $I_1'(u)=I_0(u)-u^{-1}I_1(u)$,
into the above equation, we obtain
\begin{align*}
I_1'(u)I_1(2u)-2I_1(u)I_1'(2u)&=I_0(u)I_1(2u)-2I_0(2u)I_1(u)\\
&=\frac{1}{u}\bigg(\frac{uI_0(u)}{I_{1}(u)}-\frac{2uI_0(2u)}{I_{1}(2u)}\bigg).
\end{align*}
By using a result of Simpson and Spector~\cite{MR736509} on the monotonicity
of the ratios of modified {B}essel functions $uI_v(u)/I_{v+1}(u), (v\ge 0)$,
namely, that for all $v\ge 0$, $uI_v(u)/I_{v+1}(u)$ is strictly monotone
decreasing on $(0, \infty)$, we arrive at the inequality in \eqref{eqbin}.
\end{proof}

We shall prove for all $n\ge 158$ that
$$\bigg|\cos\bigg(\frac{\pi\delta}{18}+\frac{2\pi n}{3}\bigg)\bigg|>
M(L_{\delta, n}).$$
From this and Theorem \ref{mth1} we see that $c_{\delta}(n)$
has the same sign as $\cos\Big(\frac{\pi\delta}{18}+\frac{2\pi n}{3}\Big)$,
and hence the proof of Corollary \ref{cor} follows.
This is because $\cos(\pi\delta/18+2\pi n/3)$ is a periodic function in $n$ of
period $3$, and because of Lemma \ref{lem32} above.
We just need to prove that
$$\bigg|\cos\bigg(\frac{\pi\delta}{18}+\frac{2\pi j}{3}\bigg)\bigg|>
M(L_{\delta, 158})$$
holds for all $j\in\{0,1,2\}$ and $\delta\in[0.227, 2.9999]$.
This can be verified by \emph{Mathematica}.

\section{Arithmetic properties
  of the cubic Borwein coefficients}\label{sec:van}

For convenience, we use standard $q$-series notation (cf.\ \cite{GR2004}).
For $a\in\cb$ and $0<|q|<1$, let
$$(a;q)_\infty:=\prod_{j=0}^\infty(1-a q^j),$$
and
$$
(a_1,\ldots,a_m;q)_\infty:=(a_1;q)_\infty\cdots(a_m;q)_\infty.
$$
Further let the modified Jacobi theta function be defined by
$$\theta(z;q):=(z,q/z;q)_\infty=\frac 1{(q;q)_\infty}\sum_{n\in\zb}
(-1)^n q^{\frac{n(n-1)}2}z^n,$$
where the last equation is equivalent to Jacobi's triple product
identity~\cite[Equation~(1.6.1)]{GR2004}.

In this section we are interested in arithmetic properties
(including sign patterns, vanishing and divisibility properties)
of the Borwein coefficients for exponent $\delta=3$, i.e. for
$$c_k^{(3)}(n):=\big[q^n\big]\frac{(q;q)_\infty^3}{(q^k;q^k)_\infty^3}$$
where $k>1$ is an integer,
which we shall refer to as
\emph{cubic} Borwein coefficients. (Here we relax the condition
that $k$ is a prime, which we assumed in the earlier sections.
Nevertheless, in relevant cases,
$k$ will be assumed to be odd, or even a prime.)

We deduce the arithmetic properties we are interested in from
the following result.
\begin{theorem}\label{main}
If $k$ is a positive even integer then
\begin{align*}
(q;q)_{\infty}^3&=
  \sum_{0\le \ell< \frac{k-1}{2}}(-1)^{\ell}q^{\frac{\ell(\ell+1)}{2}}
    \left(-q^{\frac{k(k-1-2\ell)}{2}},-q^{\frac{k(k+1+2\ell)}{2}},q^{k^2};
    q^{k^2}\right)_{\infty}\\*
  &\quad\;\times\left(2\ell+1-2k\sum_{n\ge 0}
    \left(\frac{q^{k(kn+\frac{k-1-2\ell}{2})}}
    {1+q^{k(kn+\frac{k-1-2\ell}{2})}}-
    \frac{q^{k(kn+\frac{k+1+2\ell}{2})}}
    {1+q^{k(kn+\frac{k+1+2\ell}{2})}}\right)\right).
\end{align*}
If $k$ is a positive odd integer then
\begin{align*}
  (q;q)_{\infty}^3&-
  (-1)^{\frac{k-1}{2}}kq^{\frac{k^2-1}{8}}(q^{k^2};q^{k^2})_{\infty}^3\\
  &=\sum_{0\le \ell< \frac{k-1}{2}}(-1)^{\ell}q^{\frac{\ell(\ell+1)}{2}}
     \left(q^{\frac{k(k-1-2\ell)}{2}},q^{\frac{k(k+1+2\ell)}{2}},q^{k^2};
     q^{k^2}\right)_{\infty}\\*
  &\quad\;\times\left(2\ell+1+2k\sum_{n\ge 0}
    \left(\frac{q^{k(kn+\frac{k-1-2\ell}{2})}}
    {1-q^{k(kn+\frac{k-1-2\ell}{2})}}-
    \frac{q^{k(kn+\frac{k+1+2\ell}{2})}}
    {1-q^{k(kn+\frac{k+1+2\ell}{2})}}\right)\right).
\end{align*}
\end{theorem}
Notice that the right-hand sides of the two identities above are finite sums
whose terms involve Jacobi triple products and Lambert series.
The following corollaries are direct consequences of Theorem~\ref{main}
whose proof we give at the end of this section.

Choosing $k=2$ in Theorem~\ref{main}, we have
\begin{align*}
  (q;q)_{\infty}^3=\left(-q,-q^3,q^{4};q^{4}\right)_{\infty}\left(1-4\sum_{n\ge 0}
  \left(\frac{q^{4n+1}}{1+q^{4n+1}}-\frac{q^{4n+3}}{1+q^{4n+3}}\right)\right).
\end{align*}
Replacing $q$ by $-q$, we get
\begin{align*}
  \frac{(-q;-q)_{\infty}^3}{\left(q,q^3,q^{4};q^{4}\right)_{\infty}}
  =1+4\sum_{n\ge 0}\left(\frac{q^{4n+1}}{1-q^{4n+1}}-
  \frac{q^{4n+3}}{1-q^{4n+3}}\right).
\end{align*}
After simplification and an application of Jacobi's triple
product identity we obtain the following classical result
(cf.\ \cite[Eq.~(3.2.8)]{B2006}).
\begin{corollary}[Sum of two squares theorem]
  For $k=2$ we have
$$
\bigg(\sum_{n=-\infty}^\infty q^{n^2}\bigg)^2
=(-q,-q,q^2;q^2)_{\infty}^2=
1+4\sum_{n\ge 0}\bigg(\frac{q^{4n+1}}{1-q^{4n+1}}-
  \frac{q^{4n+3}}{1-q^{4n+3}}\bigg).
$$
\end{corollary}
Similarly, choosing $k=3$ in Theorem~\ref{main}, we readily obtain the
following result which can be interpreted as an identity for the
cubic theta functions of the Borwein brothers~\cite{BB1991}.
\begin{corollary}[A cubic theta function addition formula]\label{cor:c-th}
For $k=3$ we have
\begin{equation}\label{eq:c-th}
\frac{(q,q)_{\infty}^3+3q(q^9;q^9)_{\infty}^3}{(q^3;q^3)_{\infty}}=
1+6\sum_{n\ge 0}\bigg(\frac{q^{9n+3}}{1-q^{9n+3}}-
\frac{q^{9n+6}}{1-q^{9n+6}}\bigg).
\end{equation}
\end{corollary}
The connection of Corollary~\ref{cor:c-th} to the cubic theta functions
is as follows: For
$$
L(q):=\sum_{n,m=-\infty}^\infty q^{n^2+nm+m^2}
$$
the Borweins, in \cite[p.~695]{BB1991} defined the following
three cubic analogues of Jacobi theta functions,
\begin{equation}\label{eq:abc}
a(q):=L(q),\quad\; b(q):=\big[3L(q)^3-L(q)\big]/2,\quad\;
c(q):=\big[L(q^{1/3}-L(q)\big]/2.
\end{equation}
(Explicit series representations for $a(q)$, $b(q)$ and $c(q)$
are conveniently listed in \cite[Equations (1.6)--(1.8)]{BBG1994}.)
Now the Lambert series for $L(q)$ (thus for $a(q)$) is
\begin{equation}\label{eq:a}
L(q)=1+6\sum_{n\ge 0}\bigg(\frac{q^{3n+1}}{1-q^{3n+1}}-
\frac{q^{3n+2}}{1-q^{3n+2}}\bigg),
\end{equation}
which is originally due to Lorenz~\cite[p.~11]{L1871}.
See \cite[p.~43]{BBG1994} for a discussion on the history of \eqref{eq:a}
including alternative proofs.
While a central result in the theory of Borweins' cubic theta functions
is the cubic identity
\cite[Equation~(2.3)]{BB1991}
$$
a(q)^3=b(q)^3+c(q)^3,
$$
many other identities that connect the three cubic theta functions $a(q)$,
$b(q)$, $c(q)$ exist in addition, including (cf.\ \cite[Equation~(3.30)]{C2017})
\begin{equation}\label{eq:abc2}
a(q^3)=b(q)+c(q^3),
\end{equation}
which is immediate from the defining relations \eqref{eq:abc}.
Now since (cf.\ \cite[Proposition~2.2]{BBG1994})
\begin{equation}
b(q)=\frac{(q;q)_\infty^3}{(q^3;q^3)_\infty},\quad\text{and}\quad
c(q)=3q^{\frac 13}\frac{(q^3;q^3)_\infty^3}{(q;q)_\infty},
\end{equation}
it is clear that \eqref{eq:c-th} is nothing else than \eqref{eq:abc2}
in explicit terms.

In the case that $k$ is an odd positive integer, the Lambert series
appearing in the statement of Theorem~\ref{main} contain only non-zero
coefficients of powers of $q$ whose exponents are multiples of $k$.
The triple product however has the prefactor $q^{\ell(\ell+1)/2}$
which is of relevance. Since $\ell(\ell+1)/2\equiv h\pmod k$ is
equivalent to $(2\ell+1)^2\equiv 1+8h\pmod k$,
the following corollary is immediate.
\begin{corollary}\label{cor13}
  Let $k$ be an odd positive integer and $h$ be a non-negative integer
  less than $k$ such that
  $$\big|\big\{\ell\!\!\!\pmod k: (2\ell+1)^2\equiv
  1+8h\!\!\!\pmod k\big\}\big|=0.$$
Then
$$
c_{k}^{(3)}(kn+h)=0,\qquad\text{for all $n\in\nb_0$.}
$$
\end{corollary}
In particular, for the cases $k=3,5,7,9$, we have
\begin{align*}
c_3^{(3)}(3n+2)&=0,\\
c_5^{(3)}(5n+2)&=c_5^{(3)}(5n+4)=0,\\
c_7^{(3)}(7n+2)&=c_7^{(3)}(7n+4)=c_7^{(3)}(7n+5)=0,\\
c_9^{(3)}(9n+2)&=c_9^{(3)}(9n+4)=c_9^{(3)}(9n+5)=c_9^{(3)}(9n+7)
=c_9^{(3)}(9n+8)=0,
\end{align*}
for all $n\in\nb_0$.

We now show that for any odd prime $p$ the cubic Borwein coefficients have
a sign pattern of period $p$, as stated in Theorem~\ref{thm:a3}.
\begin{proof}[Proof of Theorem~\ref{thm:a3}]
Notice that if $p$ is an odd prime and $0\le \ell_1<\ell_2< \frac{p-1}{2}$ then
$$\frac{\ell_1(\ell_1+1)}{2}\not\equiv \frac{\ell_2(\ell_2+1)}{2}\pmod p,$$
and for all $0\le \ell< \frac{p-1}{2}$ the expression
\begin{align*}
  \frac{(q^p,q^{\ell},q^{p-\ell};q^p)_{\infty}}{(q;q)_{\infty}^3}
  &\sum_{n\ge 0}\left(\frac{q^{pn+\ell}}{1-q^{pn+\ell}}
    -\frac{q^{pn+p-\ell}}{1-q^{pn+p-\ell}}\right)\\
  &=\frac{(q^p,q^{\ell},q^{p-\ell};q^p)_{\infty}}
    {(q;q)_{\infty}^3}\sum_{n\ge 0}\frac{q^{pn+\ell}(1-q^{p-2\ell})}
    {(1-q^{pn+\ell})(1-q^{pn+p-\ell})}\\
  &=\frac{(q^{\ell},q^{p-\ell};q^p)_{\infty}}{(q;q)_{\infty}^2}
    \sum_{n\ge 0}\frac{q^{pn+\ell}(1-q^{p-2\ell})}
    {(1-q^{pn+\ell})(1-q^{pn+p-\ell})}
    \prod_{\substack{n\ge 1\\ n\not\equiv 0\pmod p}}\frac{1}{1-q^n}
\end{align*}
is in $\nb[[q]]$, i.e., a power series in $q$ with positive coefficients.
Replacing $q$ by $q^p$ and $\ell$ by $(p-1-2\ell)/2$, the above expression
is in $\nb[[q^p]]$. Also notice that
for all odd positive integers $k$, and all integers $0\le\ell<\frac{k-1}{2}$,
$$\frac{k^2-1}{8}-\frac{\ell(\ell+1)}{2}\not\equiv 0\pmod k.$$
Theorem~\ref{thm:a3} now readily follows from Theorem~\ref{main}.
\end{proof}
As a by-product of the above proof, due to the appearance of the factor
$kq^{\frac{k^2-1}8}$ (which trivially is divisible by $k$) in the
second formula in Theorem~\ref{main}, we have the following result:
\begin{corollary}
Let $k$ be an odd positive integer. Then
$$c_k^{(3)}\bigg(kn+\frac{k^2-1}{8}\bigg)\equiv 0\pmod k,
\qquad\text{for all $n\in\nb_0$.}$$
\end{corollary}

Before we prove Theorem~\ref{main}, we give a proposition and a lemma.
\begin{proposition}\label{prop:jtpe}
For each $k\in\nb$ we have
$$
\theta(z;q)=\frac{\big(q^{k^2};q^{k^2}\big)_{\infty}}{(q;q)_{\infty}}
\sum_{\left\lceil\frac{1-k}{2}\right\rceil\le \ell\le
  \left\lceil\frac{k-1}{2}\right\rceil}
(-1)^{\ell}q^{\frac{\ell(\ell-1)}{2}}z^{\ell }
\theta\Big((-1)^{k-1}z^{k}q^{\frac{k(k-1+2\ell)}{2}};q^{k^2}\Big).
$$
\end{proposition}
\begin{proof}By Jacobi's triple product identity and
dissection of the sum into residue classes modulo $k$, we have
\begin{align*}
\theta(z;q)&=\frac{1}{(q;q)_{\infty}}
             \sum_{\left\lceil\frac{1-k}{2}\right\rceil\le \ell\le
             \left\lceil\frac{k-1}{2}\right\rceil}
             (-1)^{\ell}q^{\frac{\ell(\ell-1)}{2}}z^{\ell}
             \sum_{n\in\zb}(-1)^{kn}
             q^{k^2\frac{n(n-1)}{2}+\frac{nk(k-1)}{2}+kn\ell}z^{kn}\\
           &=\frac{\big(q^{k^2};q^{k^2}\big)_{\infty}}
             {(q;q)_{\infty}}\sum_{\left\lceil\frac{1-k}{2}\right\rceil\le
             \ell\le \left\lceil\frac{k-1}{2}\right\rceil}
(-1)^{\ell}q^{\frac{\ell(\ell-1)}{2}}z^{\ell }
\theta\Big((-1)^{k-1}z^{k}q^{\frac{k(k-1+2\ell)}{2}};q^{k^2}\Big),
\end{align*}
which completes the proof.
\end{proof}
\begin{lemma}\label{lem:der}Let $k\in\zb$ and $\alpha\in(0,1)$.
\begin{align*}
&\frac{\rd}{\rd x}\bigg|_{x=0}\theta\big((-1)^kq^{\alpha}e^{-x};q\big)\\
  &=\theta\big((-1)^kq^{\alpha};q\big)
    \sum_{n\ge 0}\bigg(\frac{(-1)^kq^{n+\alpha}}{1-(-1)^kq^{n+\alpha}}
    -\frac{(-1)^kq^{n+1-\alpha}}{1-(-1)^kq^{n+1-\alpha}}\bigg).
\end{align*}
\end{lemma}
\begin{proof}We compute
\begin{align*}
&\frac{\rd}{\rd x}\bigg|_{x=0}\frac{\theta\big((-1)^kq^{\alpha}e^{-x};q\big)}
{\theta\big((-1)^kq^{\alpha};q\big)}\\
&=\frac{\rd}{\rd x}\bigg|_{x=0}\log\theta\big((-1)^kq^{\alpha}e^{-x};q\big)\\
&=\sum_{n\ge 0}\frac{\rd}{\rd x}\bigg|_{x=0}\log
\Big(\big(1-(-1)^kq^{n+\alpha}e^{-x}\big)
\big(1-(-1)^kq^{n+1-\alpha}e^{x}\big)\Big)\\
  &=\sum_{n\ge 0}\bigg(\frac{(-1)^kq^{n+\alpha}}
    {1-(-1)^kq^{n+\alpha}}-\frac{(-1)^kq^{n+1-\alpha}}
    {1-(-1)^kq^{n+1-\alpha}}\bigg),
\end{align*}
which completes the proof.
\end{proof}
For $k=0$ and $\alpha\to 1^-$ we get from Lemma~\ref{lem:der}
\begin{equation}\label{eq:dth}
\frac{\rd}{\rd x}\bigg|_{x=0}\theta\big(e^{-x};q\big)=
(q;q)_\infty^2.
\end{equation}

After these preparations, we are ready for the proof of Theorem~\ref{main}.
For convenience, for a statement $A$ we use the notation
$$
{\bf 1}_A=\begin{cases}1&\text{if $A$ is true,}\\0&\text{otherwise.}\end{cases}
$$
\begin{proof}[Proof of Theorem \ref{main}]
Using Equation~\eqref{eq:dth}, Proposition~\ref{prop:jtpe}
and Lemma~\ref{lem:der}, we have
\begin{align*}
&(q;q)_{\infty}^2=\frac{\rd}{\rd x}\bigg|_{x=0}\theta \big(e^{-x};q\big)\\
  &=\frac{\big(q^{k^2};q^{k^2}\big)_{\infty}}
    {(q;q)_{\infty}}\frac{\rd}{\rd x}
    \bigg|_{x=0}  \sum_{\left\lceil\frac{1-k}{2}\right\rceil\le
    \ell\le \left\lceil\frac{k-1}{2}\right\rceil}
(-1)^{\ell}q^{\frac{\ell(\ell-1)}{2}}e^{-\ell x }
\theta\Big((-1)^{k-1}e^{-kx}q^{\frac{k(k-1+2\ell)}{2}};q^{k^2}\Big)\\
&=\frac{\big(q^{k^2};q^{k^2}\big)_{\infty}}{(q;q)_{\infty}}
\frac{\rd}{\rd x}\bigg|_{x=0}
\sum_{\frac{1-k}{2}<\ell\le\left\lceil\frac{k-1}{2}\right\rceil}
(-1)^{\ell}q^{\frac{\ell(\ell-1)}{2}}e^{-\ell x }
\theta\Big((-1)^{k-1}e^{-kx}q^{\frac{k(k-1+2\ell)}{2}};q^{k^2}\Big)\\*
  &\quad\;+{\bf 1}_{k\equiv 1\!\!\!\!\pmod 2}
    \frac{\big(q^{k^2};q^{k^2}\big)_{\infty}}
   {(q;q)_{\infty}}\frac{\rd}{\rd x}\bigg|_{x=0}(-1)^{\frac{1-k}{2}}
q^{\frac{k^2-1}{8}}e^{-\frac{1-k}{2} x }
\theta\Big((-1)^{k-1}e^{-kx};q^{k^2}\Big)\\
&=\frac{\big(q^{k^2};q^{k^2}\big)_{\infty}}{(q;q)_{\infty}}
  \sum_{\frac{1-k}{2}< \ell\le \left\lceil\frac{k-1}{2}\right\rceil}
   (-1)^{\ell}q^{\frac{\ell(\ell-1)}{2}}
\theta\Big((-1)^{k-1}q^{\frac{k(k-1+2\ell)}{2}};q^{k^2}\Big)\\
&\quad\;\times\Bigg(-\ell-k\sum_{n\ge 0}
\Bigg(\frac{(-1)^kq^{k^2(n+\frac{k-1+2\ell}{2k})}}
{1+(-1)^kq^{k^2(n+\frac{k-1+2\ell}{2k})}}-
\frac{(-1)^kq^{k^2(n+\frac{k+1-2\ell}{2k})}}
{1+(-1)^kq^{k^2(n+\frac{k+1-2\ell}{2k})}}\Bigg)\Bigg)\\*
  &\quad\;+{\bf 1}_{k\equiv 1\!\!\!\!\pmod 2}\,k
    \frac{\big(q^{k^2};q^{k^2}\big)_{\infty}}
    {(q;q)_{\infty}}(-1)^{\frac{1-k}{2}}q^{\frac{k^2-1}{8}}
    \big(q^{k^2};q^{k^2}\big)_\infty^2.
\end{align*}
Therefore,
\begin{align*}
&(q;q)_{\infty}^3-(-1)^{\frac{k-1}{2}}kq^{\frac{k^2-1}{8}}
\big(q^{k^2};q^{k^2}\big)_{\infty}^3\,{\bf 1}_{k\equiv 1\!\!\!\!\pmod 2}\\
&=\sum_{0\le \ell< \frac{k-1}{2}}(-1)^{\ell}q^{\frac{\ell(\ell+1)}{2}}
\Big((-1)^{k-1}q^{\frac{k(k-1-2\ell)}{2}},
(-1)^{k-1}q^{\frac{k(k+1+2\ell)}{2}},q^{k^2};q^{k^2}\Big)_{\infty}\\*
&\quad\;\times\Bigg(2\ell+1-2k\sum_{n\ge 0}\Bigg(\frac{(-1)^{k}
q^{k(kn+\frac{k-1-2\ell}{2})}}{1+(-1)^kq^{k(kn+\frac{k-1-2\ell}{2})}}-
\frac{(-1)^{k}q^{k(kn+\frac{k+1+2\ell}{2})}}
{1+(-1)^kq^{k(kn+\frac{k+1+2\ell}{2})}}\Bigg)\Bigg).
\end{align*}
This completes the proof.
\end{proof}

\appendix
\section{Further conjectures on precise sign patterns}
\label{sec:app}

Conjecture~\ref{conj1} is a statement about the
\textit{precise} (not asymptotic) sign pattern of a
$q$-series with base $q^3$.
We now present similar conjectures about precise sign patterns
for other bases $q^m$, where $m$ is a small positive integer
(we choose to list the cases $m\le 12$ here).
To keep the exposition short, we refrain from giving all the details
about justifying the specific ranges of the exponents $\delta$
for which the respective sign patterns appear to hold.
The analysis would be similar to that for Conjecture~\ref{conj1}.
We nevertheless give some hints about how the various irrational
constants emerge.

We have made similar observations for products involving other
bases ($q^m$ with selected $m>12$). While it should be possible to
approach the conjectures asymptotically by the methods
we used in this paper to treat the $m=3$ case,
we believe it to be a challenge to
prove the precise results (that depend on the respective
specified ranges), at least in the cases $m\ne 2,6$.

\subsection{Base $\boldsymbol q^{\boldsymbol 2}$}
The $q$-series coefficients of the infinite Borwein product
\begin{equation*}
G_2(q)=(q;q^2)_\infty,
\end{equation*}
evidently have the sign pattern $+-$.

We believe that even more is true:
\begin{conjecture}\label{conjm2}
The $q$-series coefficients of
$G_2(q)^\delta$ exhibit the sign pattern $+-$ for any $\delta\ge 1$.
\end{conjecture}
Since
$$
G_2(q)^\delta=1-\delta q+\frac{\delta(\delta-1)}2 q^2+O(q^3),
$$
it is clear that if $0<\delta<1$ the coefficient of $q^2$ would be
negative, in violation with the sign pattern $+-$. Further,
since $G_2(-q)=(-q;q^2)_\infty$, it is clear that Conjecture~\ref{conjm2}
is true for all positive integers $\delta$.
We speculate that a proof of Conjecture~\ref{conjm2}
for $\delta\in\rb^+\setminus\nb$
without using asymptotic machinery
is feasible (yet we have none available as for now).
A similar situation arises in the base $q^6$ case,
see Subappendix~\ref{ssec:6} below.

\subsection{Base $\boldsymbol q^{\boldsymbol 5}$}
The infinite product
\begin{equation*}
Q_5(q)=\frac{(q,q^4;q^5)_\infty}{(q^2,q^3;q^5)_\infty}
\end{equation*}
is the well-known product for the Rogers--Ramanujan continued fraction.
It was shown by Richmond and Szekeres~\cite{MR515217} that
the coefficients of $Q_5(q)$ have the sign pattern $+-+--$.

We conjecture that more is true:
\begin{conjecture}\label{conj5}
The $q$-series coefficients of
$Q_5(q)^\delta$ exhibit the sign pattern $+-+--$ for
$$1\le \delta\le \frac{\sqrt{97}-5}2\approx
2.424428900898\ldots.$$
For
$$2.571366313289\ldots\approx \alpha\le \delta\le 4,$$
they exhibit the sign pattern $+-+-+$.
(Here $\alpha$ is the unique real root of the
polynomial
$$x^7+35x^6+7x^5-6055 x^4-14336 x^3+104300 x^2-184752 x+282240$$
that satisfies $2<\alpha<3$.)
For $\delta=-1$ they exhibit the sign pattern $++---$, and for
$-3\le \delta\le -2$ the sign pattern $+++--$.
\end{conjecture}
The constant $\frac{\sqrt{97}-5}2$ comes from the coefficient
of $q^{4}$ in $Q_5(q)^\delta$, which contains $\delta^2+5\delta-18$
as a factor (of which $\frac{\sqrt{97}-5}2$ is a root).
The specific constant $\alpha$ comes from the coefficient
of $q^{9}$ in $Q_5(q)^\delta$. For $\delta=\alpha$, this coefficient
vanishes and changes its sign locally as $\delta$ traverses that point.

\subsection{Base $\boldsymbol q^{\boldsymbol 6}$}\label{ssec:6}
Here we consider the infinite product
\begin{equation*}
Q_6(q)=(q,q^5;q^6)_\infty.
\end{equation*}

\begin{conjecture}\label{conjd6}
  The $q$-series coefficients of
 $Q_6(q)^\delta$ exhibit the alternating sign pattern $(+-)^3$ for all $\delta\ge 3$.
\end{conjecture}
Since $Q_6(-q)=(-q,-q^5;q^6)_\infty$, it is clear that Conjecture~\ref{conjd6}
is true for all positive integers $\delta$. 

\subsection{Base $\boldsymbol q^{\boldsymbol 7}$}
Here we consider the infinite Borwein product
\begin{equation*}
G_7(q)=\frac{(q;q)_\infty}{(q^7;q^7)_\infty}.
\end{equation*}

\begin{conjecture}
  The $q$-series coefficients of
  $G_7(q)^\delta$ exhibit the sign pattern $+--\,0\,0+0$ for $\delta=1$.
  (The zeroes indicate vanishing.)
For $2\le \delta< 3$ they exhibit the sign pattern $+--+++-$,
for $\delta=3$ the sign pattern $+-0+0\,0\,-$, and for $3<\delta\le 5$
the sign pattern $+-++---$.
\end{conjecture}
We would like to emphasize that the periodicity of
the signs of the Borwein coefficients (for any prime $p$)
was already proved by Andrews \cite{MR1395410},
corresponding to the case $\delta=1$ in the conjecture.
We further notice that we already proved the vanishing
of the respective coefficients in the case $\delta=3$
in Corollary~\ref{cor13}.

\subsection{Base $\boldsymbol q^{\boldsymbol 8}$}
The infinite  product
\begin{equation*}
Q_8(q)=\frac{(q,q^7;q^8)_\infty}{(q^3,q^5;q^8)_\infty}
\end{equation*}
is the well-known product for the G\"ollnitz--Gordon
continued fraction. It was shown by
Hirschhorn~\cite{Hh2001} that the coefficients of
$Q_8(q)$ exhibit the sign pattern $+-0+-+0-$,
and that  the coefficients of $Q_8(q)^{-1}$ exhibit
the sign pattern $+++0---0$.

We conjecture that even more is true:
\begin{conjecture}
The $q$-series coefficients of
$Q_8(q)^\delta$ exhibit for $\delta=2$ the length $16$ sign pattern
$+-++-+--+--+-++-$. For
$$2.664479110226972\ldots\approx \beta\le \delta\le 4$$
they exhibit the sign pattern $+-++-+--$.
(Here $\beta$ is the unique real root of the
polynomial
\begin{align*}
&x^{12}-90x^{11}+ 1457 x^{10} + 30486 x^9 - 537081 x^8\\*
  &+ 1892346 x^7-{} 3683653 x^6  - 837509646 x^5 + 774767020 x^4\\*
  & +{} 3333687384 x^3- 40887173664 x^2 + 94379731200 x+49816166400
\end{align*}
that satisfies $2<\beta<3$.)
For
$$-1< \delta\le \frac{7-\sqrt{73}}2\approx -0.77200187265877\ldots$$
they exhibit the sign pattern $+++----+$. For $\delta=-2$
they exhibit the length 16 sign pattern $+++++----+++----$.
\end{conjecture}
The constant $\frac{7-\sqrt{73}}2$ comes from the coefficient
of $q^{4}$ in $Q_8(q)^\delta$, which contains $\delta^2-7\delta-6$
as a factor (of which $\frac{7-\sqrt{73}}2$ is a root).
The specific constant $\beta$ comes from the coefficient
of $q^{14}$ in $Q_8(q)^\delta$. For $\delta=\beta$, this coefficient
vanishes and changes its sign locally as $\delta$ traverses that point.

\subsection{Base $\boldsymbol q^{\boldsymbol 1\boldsymbol 0}$}
Here we consider the infinite product
\begin{equation*}
Q_{10}(q)=\frac{(q,q^9;q^{10})_\infty}{(q^3,q^7;q^{10})_\infty}.
\end{equation*}
\begin{conjecture}
  The $q$-series coefficients of
  $Q_{10}(q)^\delta$ exhibit for $\delta=1$ the sign pattern
$+-++--+--+$, and for $\delta=-1$ the sign pattern $++++-----+$.
\end{conjecture}

\subsection{Base $\boldsymbol q^{\boldsymbol 1\boldsymbol 1}$}
Here we consider the infinite Borwein product
\begin{equation*}
  G_{11}(q)=\frac{(q;q)_\infty}{(q^{11};q^{11})_\infty}.
\end{equation*}
\begin{conjecture}
  The $q$-series coefficients of
  $G_{11}(q)^\delta$ exhibit for $\delta=1$ the sign pattern
  $+--\,0-+\,0+0\,0\,0$, for
  $$1.7584535519419\ldots\approx\gamma\le \delta\le 2$$
  they exhibit the sign pattern $+--+++-+--+$.
  (Here $\gamma$ is the unique real root of the
polynomial
\begin{align*}
&x^{18}-605 x^{17}+ 157086 x^{16} - 23170380 x^{15} + 2166947862 x^{14}\\*
& - 135855285510 x^{13}  + 5889093658432 x^{12} - 179555226371060 x^{11}\\*
&+ 3882606726301473 x^{10} - 59646447279831765 x^9 + 648313198119620778 x^8\\*
&- 4932359196939174840 x^7  +  25753067609579704864 x^6\\*
&- 89277087875773607120 x^5 +  194830259522753020704 x^4\\*
  &- 246159139789631646720 x^3 + 159155369289255052800 x^2\\*
  &- 42300952112982528000 x +3243869344235520000
\end{align*}
that satisfies $1.5<\gamma<2$.)
For $\delta=3$ the coefficients exhibit the sign pattern $+-0+-\,0-0\,0\,0\,+$.
(Again, zeroes indicate that the respective coefficients vanish.)
\end{conjecture}
Notice that we already proved the vanishing of the respective coefficients
in the case $\delta=3$ in Corollary~\ref{cor13}.

The specific constant $\gamma$ comes from the coefficient
of $q^{21}$ in $G_{11}(q)^\delta$. For $\delta=\gamma$, this coefficient
vanishes and changes its sign locally as $\delta$ traverses that point.

\subsection{Base $\boldsymbol q^{\boldsymbol 1\boldsymbol 2}$}
Here we consider the infinite product
\begin{equation*}
Q_{12}(q)=\frac{(q,q^{11};q^{12})_\infty}{(q^5,q^7;q^{12})_\infty}.
\end{equation*}
\begin{conjecture}
  The $q$-series coefficients of
  $Q_{12}(q)^\delta$ exhibit for $\delta=1$ the sign pattern
  $+-+\,0-+-+-0+-$, for $2\le \delta\le 3$ they exhibit the sign pattern
  $+-+--+-+-++-$.
For $\delta=-1$ they exhibit the sign pattern $+++++\,0-----0$,
and for $-1<\delta<0$ the sign pattern $+++++------\,+$.
\end{conjecture}


\end{document}